\documentclass{amsart}

\numberwithin{equation}{section}

\usepackage{graphicx,amsthm}
\usepackage{verbatim,amsmath,amscd,amssymb,color}
\usepackage[mathscr]{eucal}
\input xy
\xyoption{all}
\begin{document}
\renewcommand{\labelenumi}{$($\roman{enumi}$)$}
\renewcommand{\labelenumii}{$(${\rm \alph{enumii}}$)$}
\font\germ=eufm10
\newcommand{\cI}{{\mathcal I}}
\newcommand{\cA}{{\mathcal A}}
\newcommand{\cB}{{\mathcal B}}
\newcommand{\cC}{{\mathcal C}}
\newcommand{\cD}{{\mathcal D}}
\newcommand{\cE}{{\mathcal E}}
\newcommand{\cF}{{\mathcal F}}
\newcommand{\cG}{{\mathcal G}}
\newcommand{\cH}{{\mathcal H}}
\newcommand{\cK}{{\mathcal K}}
\newcommand{\cL}{{\mathcal L}}
\newcommand{\cM}{{\mathcal M}}
\newcommand{\cN}{{\mathcal N}}
\newcommand{\cO}{{\mathcal O}}
\newcommand{\cR}{{\mathcal R}}
\newcommand{\cS}{{\mathcal S}}
\newcommand{\cV}{{\mathcal V}}
\newcommand{\cX}{{\mathcal X}}
\newcommand{\fra}{\mathfrak a}
\newcommand{\frb}{\mathfrak b}
\newcommand{\frc}{\mathfrak c}
\newcommand{\frd}{\mathfrak d}
\newcommand{\fre}{\mathfrak e}
\newcommand{\frf}{\mathfrak f}
\newcommand{\frg}{\mathfrak g}
\newcommand{\frh}{\mathfrak h}
\newcommand{\fri}{\mathfrak i}
\newcommand{\frj}{\mathfrak j}
\newcommand{\frk}{\mathfrak k}
\newcommand{\frI}{\mathfrak I}
\newcommand{\fm}{\mathfrak m}
\newcommand{\frn}{\mathfrak n}
\newcommand{\frp}{\mathfrak p}
\newcommand{\fq}{\mathfrak q}
\newcommand{\frr}{\mathfrak r}
\newcommand{\frs}{\mathfrak s}
\newcommand{\frt}{\mathfrak t}
\newcommand{\fru}{\mathfrak u}
\newcommand{\frA}{\mathfrak A}
\newcommand{\frB}{\mathfrak B}
\newcommand{\frF}{\mathfrak F}
\newcommand{\frG}{\mathfrak G}
\newcommand{\frH}{\mathfrak H}
\newcommand{\frJ}{\mathfrak J}
\newcommand{\frN}{\mathfrak N}
\newcommand{\frP}{\mathfrak P}
\newcommand{\frT}{\mathfrak T}
\newcommand{\frU}{\mathfrak U}
\newcommand{\frV}{\mathfrak V}
\newcommand{\frX}{\mathfrak X}
\newcommand{\frY}{\mathfrak Y}
\newcommand{\frZ}{\mathfrak Z}
\newcommand{\rA}{\mathrm{A}}
\newcommand{\rC}{\mathrm{C}}
\newcommand{\rd}{\mathrm{d}}
\newcommand{\rB}{\mathrm{B}}
\newcommand{\rD}{\mathrm{D}}
\newcommand{\rE}{\mathrm{E}}
\newcommand{\rH}{\mathrm{H}}
\newcommand{\rK}{\mathrm{K}}
\newcommand{\rL}{\mathrm{L}}
\newcommand{\rM}{\mathrm{M}}
\newcommand{\rN}{\mathrm{N}}
\newcommand{\rR}{\mathrm{R}}
\newcommand{\rT}{\mathrm{T}}
\newcommand{\rZ}{\mathrm{Z}}
\newcommand{\bbA}{\mathbb A}
\newcommand{\bbB}{\mathbb B}
\newcommand{\bbC}{\mathbb C}
\newcommand{\bbG}{\mathbb G}
\newcommand{\bbF}{\mathbb F}
\newcommand{\bbH}{\mathbb H}
\newcommand{\bbP}{\mathbb P}
\newcommand{\bbN}{\mathbb N}
\newcommand{\bbQ}{\mathbb Q}
\newcommand{\bbR}{\mathbb R}
\newcommand{\bbV}{\mathbb V}
\newcommand{\bbZ}{\mathbb Z}
\newcommand{\adj}{\operatorname{adj}}
\newcommand{\Ad}{\mathrm{Ad}}
\newcommand{\Ann}{\mathrm{Ann}}
\newcommand{\rcris}{\mathrm{cris}}
\newcommand{\ch}{\mathrm{ch}}
\newcommand{\coker}{\mathrm{coker}}
\newcommand{\diag}{\mathrm{diag}}
\newcommand{\Diff}{\mathrm{Diff}}
\newcommand{\Dist}{\mathrm{Dist}}
\newcommand{\rDR}{\mathrm{DR}}
\newcommand{\ev}{\mathrm{ev}}
\newcommand{\Ext}{\mathrm{Ext}}
\newcommand{\cExt}{\mathcal{E}xt}
\newcommand{\fin}{\mathrm{fin}}
\newcommand{\Frac}{\mathrm{Frac}}
\newcommand{\GL}{\mathrm{GL}}
\newcommand{\Hom}{\mathrm{Hom}}
\newcommand{\hd}{\mathrm{hd}}
\newcommand{\rht}{\mathrm{ht}}
\newcommand{\id}{\mathrm{id}}
\newcommand{\im}{\mathrm{im}}
\newcommand{\inc}{\mathrm{inc}}
\newcommand{\ind}{\mathrm{ind}}
\newcommand{\coind}{\mathrm{coind}}
\newcommand{\Lie}{\mathrm{Lie}}
\newcommand{\Max}{\mathrm{Max}}
\newcommand{\mult}{\mathrm{mult}}
\newcommand{\op}{\mathrm{op}}
\newcommand{\ord}{\mathrm{ord}}
\newcommand{\pt}{\mathrm{pt}}
\newcommand{\qt}{\mathrm{qt}}
\newcommand{\rad}{\mathrm{rad}}
\newcommand{\res}{\mathrm{res}}
\newcommand{\rgt}{\mathrm{rgt}}
\newcommand{\rk}{\mathrm{rk}}
\newcommand{\SL}{\mathrm{SL}}
\newcommand{\soc}{\mathrm{soc}}
\newcommand{\Spec}{\mathrm{Spec}}
\newcommand{\St}{\mathrm{St}}
\newcommand{\supp}{\mathrm{supp}}
\newcommand{\Tor}{\mathrm{Tor}}
\newcommand{\Tr}{\mathrm{Tr}}
\newcommand{\wt}{\mathrm{wt}}
\newcommand{\Ab}{\mathbf{Ab}}
\newcommand{\Alg}{\mathbf{Alg}}
\newcommand{\Grp}{\mathbf{Grp}}
\newcommand{\Mod}{\mathbf{Mod}}
\newcommand{\Sch}{\mathbf{Sch}}\newcommand{\bfmod}{{\bf mod}}
\newcommand{\Qc}{\mathbf{Qc}}
\newcommand{\Rng}{\mathbf{Rng}}
\newcommand{\Top}{\mathbf{Top}}
\newcommand{\Var}{\mathbf{Var}}
\newcommand{\gromega}{\langle\omega\rangle}
\newcommand{\lbr}{\begin{bmatrix}}
\newcommand{\rbr}{\end{bmatrix}}
\newcommand{\forb}{\bigcirc\kern-2.8ex \because}
\newcommand{\forbb}{\bigcirc\kern-3.0ex \because}
\newcommand{\forbbb}{\bigcirc\kern-3.1ex \because}
\newcommand{\cd}{commutative diagram }
\newcommand{\SpS}{spectral sequence}
\newcommand\C{\mathbb C}
\newcommand\hh{{\hat{H}}}
\newcommand\eh{{\hat{E}}}
\newcommand\F{\mathbb F}
\newcommand\fh{{\hat{F}}}
\newcommand\Z{{\mathbb Z}}
\newcommand\Zn{\Z_{\geq0}}
\newcommand\et[1]{\tilde{e}_{#1}}
\newcommand\ft[1]{\tilde{f}_{#1}}

\def\ge{\frg}
\def\AA{{\mathcal A}}
\def\al{\alpha}
\def\bq{B_q(\ge)}
\def\bqm{B_q^-(\ge)}
\def\bqz{B_q^0(\ge)}
\def\bqp{B_q^+(\ge)}
\def\beneme{\begin{enumerate}}
\def\beq{\begin{equation}}
\def\beqn{\begin{eqnarray}}
\def\beqnn{\begin{eqnarray*}}
\def\bigsl{{\hbox{\fontD \char'54}}}
\def\bbra#1,#2,#3{\left\{\begin{array}{c}\hspace{-5pt}
#1;#2\\ \hspace{-5pt}#3\end{array}\hspace{-5pt}\right\}}
\def\cd{\cdots}
\def\CC{\mathbb{C}}
\def\CBL{\cB_L(\TY(B,1,n+1))}
\def\CBM{\cB_M(\TY(B,1,n+1))}
\def\CVL{\cV_L(\TY(D,1,n+1))}
\def\CVM{\cV_M(\TY(D,1,n+1))}
\def\ddd{\hbox{\germ D}}
\def\del{\delta}
\def\Del{\Delta}
\def\Delr{\Delta^{(r)}}
\def\Dell{\Delta^{(l)}}
\def\Delb{\Delta^{(b)}}
\def\Deli{\Delta^{(i)}}
\def\Delre{\Delta^{\rm re}}
\def\ei{e_i}
\def\eit{\tilde{e}_i}
\def\eneme{\end{enumerate}}
\def\ep{\epsilon}
\def\eeq{\end{equation}}
\def\eeqn{\end{eqnarray}}
\def\eeqnn{\end{eqnarray*}}
\def\fit{\tilde{f}_i}
\def\FF{{\rm F}}
\def\ft{\tilde{f}_}
\def\gau#1,#2{\left[\begin{array}{c}\hspace{-5pt}#1\\
\hspace{-5pt}#2\end{array}\hspace{-5pt}\right]}
\def\gl{\hbox{\germ gl}}
\def\hom{{\hbox{Hom}}}
\def\ify{\infty}
\def\io{\iota}
\def\kp{k^{(+)}}
\def\km{k^{(-)}}
\def\llra{\relbar\joinrel\relbar\joinrel\relbar\joinrel\rightarrow}
\def\lan{\langle}
\def\lar{\longrightarrow}
\def\max{{\rm max}}
\def\lm{\lambda}
\def\Lm{\Lambda}
\def\mapright#1{\smash{\mathop{\longrightarrow}\limits^{#1}}}
\def\Mapright#1{\smash{\mathop{\Longrightarrow}\limits^{#1}}}
\def\mm{{\bf{\rm m}}}
\def\nd{\noindent}
\def\nn{\nonumber}
\def\nnn{\hbox{\germ n}}
\def\catob{{\mathcal O}(B)}
\def\oint{{\mathcal O}_{\rm int}(\ge)}
\def\ot{\otimes}
\def\op{\oplus}
\def\opi{\ovl\pi_{\lm}}
\def\osigma{\ovl\sigma}
\def\ovl{\overline}
\def\plm{\Psi^{(\lm)}_{\io}}
\def\qq{\qquad}
\def\q{\quad}
\def\qed{\hfill\framebox[2mm]{}}
\def\QQ{\mathbb Q}
\def\qi{q_i}
\def\qii{q_i^{-1}}
\def\ra{\rightarrow}
\def\ran{\rangle}
\def\rlm{r_{\lm}}
\def\ssl{\hbox{\germ sl}}
\def\slh{\widehat{\ssl_2}}
\def\ti{t_i}
\def\tii{t_i^{-1}}
\def\til{\tilde}
\def\tm{\times}
\def\tt{\frt}
\def\TY(#1,#2,#3){#1^{(#2)}_{#3}}
\def\ua{U_{\AA}}
\def\ue{U_{\vep}}
\def\uq{U_q(\ge)}
\def\uqp{U'_q(\ge)}
\def\ufin{U^{\rm fin}_{\vep}}
\def\ufinp{(U^{\rm fin}_{\vep})^+}
\def\ufinm{(U^{\rm fin}_{\vep})^-}
\def\ufinz{(U^{\rm fin}_{\vep})^0}
\def\uqm{U^-_q(\ge)}
\def\uqmq{{U^-_q(\ge)}_{\bf Q}}
\def\uqpm{U^{\pm}_q(\ge)}
\def\uqq{U_{\bf Q}^-(\ge)}
\def\uqz{U^-_{\bf Z}(\ge)}
\def\ures{U^{\rm res}_{\AA}}
\def\urese{U^{\rm res}_{\vep}}
\def\uresez{U^{\rm res}_{\vep,\ZZ}}
\def\util{\widetilde\uq}
\def\uup{U^{\geq}}
\def\ulow{U^{\leq}}
\def\bup{B^{\geq}}
\def\blow{\ovl B^{\leq}}
\def\vep{\varepsilon}
\def\vp{\varphi}
\def\vpi{\varphi^{-1}}
\def\VV{{\mathcal V}}
\def\xii{\xi^{(i)}}
\def\Xiioi{\Xi_{\io}^{(i)}}
\def\W1{W(\varpi_1)}
\def\WW{{\mathcal W}}
\def\wt{{\rm wt}}
\def\wtil{\widetilde}
\def\what{\widehat}
\def\wpi{\widehat\pi_{\lm}}
\def\ZZ{\mathbb Z}
\def\RR{\mathbb R}

\def\m@th{\mathsurround=0pt}
\def\fsquare(#1,#2){
\hbox{\vrule$\hskip-0.4pt\vcenter to #1{\normalbaselines\m@th
\hrule\vfil\hbox to #1{\hfill$\scriptstyle #2$\hfill}\vfil\hrule}$\hskip-0.4pt
\vrule}}

\theoremstyle{definition}
\newtheorem{df}{Definition}[section]
\newtheorem{thm}[df]{Theorem}
\newtheorem{pro}[df]{Proposition}
\newtheorem{lem}[df]{Lemma}
\newtheorem{ex}[df]{Example}
\newtheorem{cor}[df]{Corollary}
\newtheorem{conj}[df]{Conjecture}

\newcommand{\cmt}{\marginpar}
\newcommand{\seteq}{\mathbin{:=}}
\newcommand{\cl}{\colon}
\newcommand{\be}{\begin{enumerate}}
\newcommand{\ee}{\end{enumerate}}
\newcommand{\bnum}{\be[{\rm (i)}]}
\newcommand{\enum}{\ee}
\newcommand{\ro}{{\rm(}}
\newcommand{\rf}{{\rm)}}
\newcommand{\set}[2]{\left\{#1\,\vert\,#2\right\}}
\newcommand{\sbigoplus}{{\mbox{\small{$\bigoplus$}}}}
\newcommand{\ba}{\begin{array}}
\newcommand{\ea}{\end{array}}
\newcommand{\on}{\operatorname}
\newcommand{\eq}{\begin{eqnarray}}
\newcommand{\eneq}{\end{eqnarray}}
\newcommand{\hs}{\hspace*}

\title[Ultra-discretization 
of the $\TY(A,1,n)$-Geometric Crystal]
{ $\TY(A,1,n)$-Geometric Crystal corresponding to
Dynkin index $i=2$ and its ultra-discretization}
\author{Kailash C. M\textsc{isra}}
\address{Department of Mathematics,
North Carolina State University, Raleigh, NC 27695-8205, USA}
\email{misra@ncsu.edu}
\author{Toshiki N\textsc{akashima}}
\address{Department of Mathematics, 
Sophia University, Kioicho 7-1, Chiyoda-ku, Tokyo 102-8554,
Japan}
\email{toshiki@sophia.ac.jp}
\thanks{KCM: supported in part by Simon Foundation Grant 208092 and NSA Grant H98230-12-1-0248.
TN: supported in part by JSPS Grants in Aid for 
Scientific Research $\sharp 22540031$.}

\subjclass[2000]{Primary 17B37; 17B67; Secondary 22E65; 14M15}
\date{}

\dedicatory{Dedicated to Professor Michio Jimbo on the occasion of 
his 60th birthday}

\keywords{geometric crystal, perfect crystal, 
ultra-discretization. }

\begin{abstract}
Let $\ge$ be an affine Lie algebra with index set $I = \{0, 1, 2, \cdots , n\}$ and $\ge^L$ be its Langlands dual. It is conjectured in \cite{KNO} that  for each $i \in I \setminus \{0\}$ the affine Lie algebra $\ge$ has a positive geometric crystal whose ultra-discretization is isomorphic to the limit of certain coherent family of perfect crystals for $\ge^L$. We prove this conjecture for $i=2$ and $\ge =  \TY(A,1,n)$. 

\end{abstract}

\maketitle
\renewcommand{\thesection}{\arabic{section}}
\section{Introduction}
\setcounter{equation}{0}
\renewcommand{\theequation}{\thesection.\arabic{equation}}

Let $A= (a_{ij})_{i,j \in I}, I = \{0, 1, \cdots , n\}$ be an affine
Cartan matrix and $(A, \{\al_i\}_{i \in I}, $\\
$\{\al^\vee_i\}_{\i\in I})$ be a given Cartan datum. 
Let $\ge = \ge(A)$ denote the associated affine Lie algebra \cite{Kac} 
and $U_q(\ge)$ denote the corresponding quantum affine algebra. 
Let $P= \Z \Lambda_0 \oplus \Z \Lambda_1\oplus 
\cdots \oplus \Z \Lambda_n \oplus \Z\delta$ and  
$P^\vee = \Z \al^\vee_0 \oplus \Z \al^\vee_1 \oplus \cdots \oplus 
\Z \al^\vee_n \oplus \Z d $ denote the affine weight lattice 
and the dual affine weight lattice respectively. 
For a dominant weight  $\lambda \in P^+ = \{\mu \in P \mid \mu (h_i) 
\geq 0 \quad  {\rm for \ \  all} \quad i \in I \}$ of level 
$l = \lambda (c)$ ($c =$ canonical central element), 
Kashiwara defined the crystal base $(L(\lambda), B(\lambda))$
\cite{Kas1} 
for the integrable highest weight $U_q(\ge)$-module $V(\lambda)$. 
The crystal $B(\lambda)$ is the $q= 0$ limit of the canonical basis 
\cite{Lu} or the global crystal basis \cite{Kas2}. 
It has many interesting combinatorial properties. 
To give explicit realization of the crystal $B(\lambda)$, 
the notion of affine crystal and perfect crystal has been introduced 
in \cite{KMN1}. In particular, it is shown in \cite{KMN1} that 
the affine crystal $B(\lambda)$ for the level $l \in \Z_{>0}$ 
integrable highest weight $U_q(\ge)$-module $V(\lambda)$ can be 
realized as the semi-infinite tensor product $\cdots \otimes B_l \otimes
B_l \otimes B_l$, 
where $B_l$ is a perfect crystal of level $l$. 
This is known as the path realization. 
Subsequently it is noticed in \cite{KKM} that one needs 
a coherent family of perfect crystals $\{B_l\}_{l \geq 1}$ 
in order to give a path realization of the Verma module $M(\lambda)$ 
( or $U_q^-(\ge)$). In particular, 
the crystal $B(\infty)$ of $U_q^-(\ge)$ can be realized as the
semi-infinite tensor product $\cdots \otimes B_{\infty} \otimes 
B_{\infty} \otimes B_{\infty}$ where $B_{\infty}$ is the limit of 
the coherent family of perfect crystals $\{B_l\}_{l \geq 1}$ (see
\cite{KKM}). 
At least one coherent family $\{B_l\}_{l \geq 1}$ of perfect crystals 
and its limit is known for $\ge = A_n^{(1)}, 
B_n^{(1)}, C_n^{(1)}, D_n^{(1)}, A_{2n-1}^{(2)}, A_{2n}^{(2)}, 
D_{n+1}^{(2)}, D_4^{(3)}, G_2^{(1)} $ 
(see \cite{KMN2,KKM,Y,KMOY,MMO}).

A perfect crystal is indeed a crystal for certain finite dimensional 
module called Kirillov-Reshetikhin module (KR-module for short) 
of the quantum affine algebra $U_q(\ge)$ (\cite{KR}, \cite{HKOTY,HKOTT}). 
The KR-modules are parametrized by two integers $(i, l)$, where 
$i \in I \setminus \{0\}$ and $l$ any positive integer. 
Let $\{\varpi_i\}_{i\in I\setminus\{0\}}$ be the set of level $0$ 
fundamental weights \cite{K0} . Hatayama et al  
(\cite{HKOTY,HKOTT}) conjectured that any KR-module
$W(l\varpi_i)$ 
admit a crystal base $B^{i,l}$ in the sense of Kashiwara 
and furthermore $B^{i,l}$ is perfect if $l$ is a multiple of 
$c_i^\vee\seteq \mathrm{max }(1,\frac{2}{(\al_i,\al_i)})$. 
This conjecture has been proved for quantum affine algebras 
$U_q(\ge)$ of classical types (\cite{OS,FOS1,FOS2}). 
When $\{B^{i,l}\}_{l\geq 1}$ is a coherent family of perfect crystals 
we denote its limit by $B_\infty (\varpi_i)$ (or just $B_\infty$ if there is no confusion).

On the other hand the notion of geometric crystal is introduced in
\cite{BK} 
as a geometric analog to Kashiwara's crystal (or algebraic crystal)
\cite{Kas1}. 
In fact, geometric crystal is defined in \cite{BK} for reductive
algebraic groups 
and is extended to general Kac-Moody groups in \cite{N}. 
For a given Cartan datum $(A, \{\alpha_i\}_{i \in I},
\{\al^\vee_i\}_{\i\in I} )$, 
the geometric crystal is defined as a quadruple 
$\cV(\ge)=(X, \{e_i\}_{i \in I}, \{\gamma_i\}_{i \in I}, 
\{\vep_i\}_{i\in I})$, 
where $X$ is an algebraic variety,  $e_i:\bbC^\times\times
X\longrightarrow X$ 
are rational $\bbC^\times$-actions and  
$\gamma_i,\vep_i:X\longrightarrow 
\bbC$ $(i\in I)$ are rational functions satisfying certain conditions  
( see Definition \ref{def-gc}). 
A geometric 
crystal is said to be a positive geometric crystal 
if it admits a positive structure (see Definition 2.5).
A remarkable relation between positive geometric crystals 
and algebraic crystals is the ultra-discretization functor $\mathcal
{UD}$ 
between them (see Section 2.4). Applying this functor, positive rational 
functions are transfered to piecewise linear 
functions by the simple correspondence:
$$
x \times y \longmapsto x+y, \qquad \frac{x}{y} \longmapsto x - y, 
\qquad x + y \longmapsto {\rm max}\{x, y\}.
$$

It was conjectured in \cite{KNO} that for each affine Lie algebra $\ge$ and 
each Dynkin index $i \in I \setminus {0}$, there exists a positive geometric crystal
$\cV(\ge)=(X, \{e_i\}_{i \in I}, \{\gamma_i\}_{i \in I}, 
\{\vep_i\}_{i\in I})$ whose ultra-discretization $\mathcal{UD}(\cV)$ is isomorphic 
to the limit $B_{\infty}$ of a coherent family of perfect crystals for the Langlands dual $\ge^L$.
In \cite{KNO},  it has been shown that this conjecture is true for $i=1$ and $\ge = A_n^{(1)}, 
B_n^{(1)}, C_n^{(1)}, D_n^{(1)}, A_{2n-1}^{(2)}, A_{2n}^{(2)},
D_{n+1}^{(2)}$. In \cite{N3} (resp. \cite{IN}) 
a positive geometric crystal for $\ge = G_2^{(1)}$ (resp. $\ge = D_4^{(3)}$) and $i=1$ has been
constructed and it is shown in \cite{N4} (resp. \cite{IMN}) that 
the  ultra-discretization of this positive geometric crystal is 
isomorphic to the limit of a coherent family of perfect crystals for 
$\ge ^L= D_4^{(3)}$ (resp. $\ge^L = G_2^{(1)}$) given in \cite{KMOY} (resp. \cite{MMO}). 

In this paper we have constructed a positive geometric crystal associated
with the Dynkin index $i=2$ for the affine Lie algebra $A_n^{(1)}$ and have
proved that its ultra-discretization is isomorphic to the limit $B^{2,\infty}$
of the coherent family of perfect crystals $\{B^{2,l}\}_{l \geq 1}$ for the affine 
Lie algebra $A_n^{(1)}$ given in (\cite{KMN2,OSS}).

This paper is organized as follows. In Section 2, 
we recall necessary definitions and facts about geometric crystals. 
In Section 3, we recall from \cite{OSS} (see also \cite{KMN2}) the coherent family of perfect crystals 
$\{B^{2,l}\}_{l \geq 1}$for $\ge = A_n^{(1)}$ and its limit $B^{2,\infty}$. In Sections 4, we construct a
positive affine geometric crystal $\cV =\cV(A_n^{(1)})$ explicitly. 
In Section 5, we prove that the ultra-discretization $\cX=\mathcal{UD}(\cV)$ is isomorphic
to the limit $B^{2,\infty}$ which proves the conjecture in (\cite{KNO}, Conjecture 1.2) for $i=2$ and $\ge =  \TY(A,1,n)$. 

\newpage
\renewcommand{\thesection}{\arabic{section}}
\section{Geometric crystals}
\setcounter{equation}{0}
\renewcommand{\theequation}{\thesection.\arabic{equation}}

In this section, 
we review Kac-Moody groups and geometric crystals
following  \cite{BK,Ku2,N,PK}.
\subsection{Kac-Moody algebras and Kac-Moody groups}
\label{KM}
Fix a symmetrizable generalized Cartan matrix
 $A=(a_{ij})_{i,j\in I}$ with a finite index set $I$.
Let $(\tt,\{\al_i\}_{i\in I},\\
\{\al^\vee_i\}_{i\in I})$ 
be the associated
root data, where ${\tt}$ is a vector space 
over $\bbC$ and
$\{\al_i\}_{i\in I}\subset\tt^*$ and 
$\{\al^\vee_i\}_{i\in I}\subset\tt$
are linearly independent 
satisfying $\al_j(\al^\vee_i)=a_{ij}$.

The Kac-Moody Lie algebra $\ge=\ge(A)$ associated with $A$
is the Lie algebra over $\bbC$ generated by $\tt$, the 
Chevalley generators $e_i$ and $f_i$ $(i\in I)$
with the usual defining relations (\cite{KP,PK}).
There is the root space decomposition 
$\ge=\bigoplus_{\al\in \tt^*}\ge_{\al}$.
Denote the set of roots by 
$\Delta:=\{\al\in \tt^*|\al\ne0,\,\,\ge_{\al}\ne(0)\}$.
Set $Q=\sum_i\bbZ \al_i$, $Q_+=\sum_i\bbZ_{\geq0} \al_i$,
$Q^\vee:=\sum_i\bbZ \al^\vee_i$
and $\Delta_+:=\Delta\cap Q_+$.
An element of $\Delta_+$ is called 
a {\it positive root}.
Let $P\subset \tt^*$ be a weight lattice such that 
$\bbC\ot P=\tt^*$, whose element is called a
weight.

Define simple reflections $s_i\in{\rm Aut}(\tt)$ $(i\in I)$ by
$s_i(h):=h-\al_i(h)\al^\vee_i$, which generate the Weyl group $W$.
It induces the action of $W$ on $\tt^*$ by
$s_i(\lm):=\lm-\lm(\al^\vee_i)\al_i$.
Set $\Delre:=\{w(\al_i)|w\in W,\,\,i\in I\}$, whose element 
is called a real root.

Let $\ge'$ be the derived Lie algebra 
of $\ge$ and let 
$G$ be the Kac-Moody group associated 
with $\ge'$(\cite{PK}).
Let $U_{\al}:=\exp\ge_{\al}$ $(\al\in \Delre)$
be the one-parameter subgroup of $G$.
The group $G$ is generated by $U_{\al}$ $(\al\in \Delre)$.
Let $U^{\pm}$ be the subgroup generated by $U_{\pm\al}$
($\al\in \Delre_+=\Delre\cap Q_+$), {\it i.e.,}
$U^{\pm}:=\lan U_{\pm\al}|\al\in\Del^{\rm re}_+\ran$.

For any $i\in I$, there exists a unique homomorphism;
$\phi_i:SL_2(\bbC)\rightarrow G$ such that
\[
\hspace{-2pt}\phi_i\left(
\left(
\begin{array}{cc}
c&0\\
0&c^{-1}
\end{array}
\right)\right)=c^{\al^\vee_i},\,
\phi_i\left(
\left(
\begin{array}{cc}
1&t\\
0&1
\end{array}
\right)\right)=\exp(t e_i),\,
 \phi_i\left(
\left(
\begin{array}{cc}
1&0\\
t&1
\end{array}
\right)\right)=\exp(t f_i).
\]
where $c\in\bbC^\times$ and $t\in\bbC$.
Set $\al^\vee_i(c):=c^{\al^\vee_i}$,
$x_i(t):=\exp{(t e_i)}$, $y_i(t):=\exp{(t f_i)}$, 
$G_i:=\phi_i(SL_2(\bbC))$,
$T_i:=\phi_i(\{{\rm diag}(c,c^{-1})\vert 
c\in\bbC^{\vee}\})$ 
and 
$N_i:=N_{G_i}(T_i)$. Let
$T$ (resp. $N$) be the subgroup of $G$ 
with the Lie algebra $\tt$
(resp. generated by the $N_i$'s), 
which is called a {\it maximal torus} in $G$, and let
$B^{\pm}=U^{\pm}T$ be the Borel subgroup of $G$.
We have the isomorphism
$\phi:W\mapright{\sim}N/T$ defined by $\phi(s_i)=N_iT/T$.
An element $\ovl s_i:=x_i(-1)y_i(1)x_i(-1)
=\phi_i\left(
\left(
\begin{array}{cc}
0&\pm1\\
\mp1&0
\end{array}
\right)\right)$ is in 
$N_G(T)$, which is a representative of 
$s_i\in W=N_G(T)/T$. 

\subsection{Geometric crystals}
Let $X$ be an ind-variety , 
{$\gamma_i:X\rightarrow \bbC$} and 
$\vep_i:X\longrightarrow \bbC$ ($i\in I$) 
rational functions on $X$, and
{$e_i:\bbC^\times \times X\longrightarrow X$}
$((c,x)\mapsto e^c_i(x))$ a
rational $\bbC^\times$-action.

\begin{df}
\label{def-gc}
A quadruple $(X,\{e_i\}_{i\in I},\{\gamma_i,\}_{i\in I},
\{\vep_i\}_{i\in I})$ is a 
$G$ (or $\ge$)-\\{\it geometric} {\it crystal} 
if
\begin{enumerate}
\item
$\{1\}\times X\subset dom(e_i)$ 
for any $i\in I$.
\item
$\gamma_j(e^c_i(x))=c^{a_{ij}}\gamma_j(x)$.
\item $e_i$'s satisfy the following relations.
\[
 \begin{array}{lll}
&\hspace{-20pt}e^{c_1}_{i}e^{c_2}_{j}
=e^{c_2}_{j}e^{c_1}_{i}&
{\rm if }\,\,a_{ij}=a_{ji}=0,\\
&\hspace{-20pt} e^{c_1}_{i}e^{c_1c_2}_{j}e^{c_2}_{i}
=e^{c_2}_{j}e^{c_1c_2}_{i}e^{c_1}_{j}&
{\rm if }\,\,a_{ij}=a_{ji}=-1,\\
&\hspace{-20pt}
e^{c_1}_{i}e^{c^2_1c_2}_{j}e^{c_1c_2}_{i}e^{c_2}_{j}
=e^{c_2}_{j}e^{c_1c_2}_{i}e^{c^2_1c_2}_{j}e^{c_1}_{i}&
{\rm if }\,\,a_{ij}=-2,\,
a_{ji}=-1,\\
&\hspace{-20pt}
e^{c_1}_{i}e^{c^3_1c_2}_{j}e^{c^2_1c_2}_{i}
e^{c^3_1c^2_2}_{j}e^{c_1c_2}_{i}e^{c_2}_{j}
=e^{c_2}_{j}e^{c_1c_2}_{i}e^{c^3_1c^2_2}_{j}e^{c^2_1c_2}_{i}
e^{c^3_1c_2}_je^{c_1}_i&
{\rm if }\,\,a_{ij}=-3,\,
a_{ji}=-1,
\end{array}
\]
\item
$\vep_i(e_i^c(x))=c^{-1}\vep_i(x)$ and $\vep_i(e_j^c(x))=\vep_i(x)$ if 
$a_{i,j}=a_{j,i}=0$.
\end{enumerate}
\end{df}

The condition (iv) is slightly modified from the one in 
\cite{IN,N3,N4}.

Let $W$ be the  Weyl group associated with $\ge$. 
For $w \in W$ define $R(w)$  by
\[
 R(w):=\{(i_1,i_2,\cd,i_l)\in I^l|w=s_{i_1}s_{i_2}\cd s_{i_l}\},
\]
where $l$ is the length of $w$.
Then $R(w)$ is the set of reduced words of $w$.
For a word ${\bf i}=(i_1,\cd,i_l)\in R(w)$ 
$(w\in W)$, set 
$\al^{(j)}:=s_{i_l}\cd s_{i_{j+1}}(\al_{i_j})$ 
$(1\leq j\leq l)$ and 
\begin{eqnarray*}
e_{\bf i}:&T\times X\rightarrow &X\\
&(t,x)\mapsto &e_{\bf i}^t(x):=e_{i_1}^{\al^{(1)}(t)}
e_{i_2}^{\al^{(2)}(t)}\cd e_{i_l}^{\al^{(l)}(t)}(x).
\label{tx}
\end{eqnarray*}
Note that the condition (iii) above is 
equivalent to the following:
{$e_{\bf i}=e_{\bf i'}$}
for any 
$w\in W$, ${\bf i}$.
${\bf i'}\in R(w)$.

\subsection{Geometric crystal on Schubert cell}
\label{schubert}

Let $w\in W$ be a Weyl group element and take a 
reduced expression $w=s_{i_1}\cd s_{i_l}$. 
Let $X:=G/B$ be the flag
variety, which is an ind-variety 
and $X_w\subset X$ the
Schubert cell associated with $w$, which has 
a natural geometric crystal structure
(\cite{BK,N}).
For ${\bf i}:=(i_1,\cd,i_k)$, set 
\begin{equation}
B_{\bf i}^-
:=\{Y_{\bf i}(c_1,\cd,c_k)
:=Y_{i_1}(c_1)\cd Y_{i_k}(c_k)
\,\vert\, c_1\cd,c_k\in\bbC^\times\}\subset B^-,
\label{bw1}
\end{equation}
where $Y_i(c):=y_i(\frac{1}{c})\al^\vee_i(c)$.
If $I=\{i_1,\cd,i_k\}$, this has a geometric crystal structure(\cite{N})
isomorphic to $X_w$. 
The explicit forms of the action $e^c_i$, the rational 
function $\vep_i$  and $\gamma_i$ on 
$B_{\bf i}^-$ are given by
\begin{eqnarray}
&& e_i^c(Y_{\bf i}(c_1,\cd,c_k))
=Y_{\bf i}({\mathcal C}_1,\cd,{\mathcal C}_k)),\nn \\
&&\text{where}\nn\\
&&{\mathcal C}_j:=
c_j\cdot \frac{\displaystyle \sum_{1\leq m\leq j,i_m=i}
 \frac{c}
{c_1^{a_{i_1,i}}\cd c_{m-1}^{a_{i_{m-1},i}}c_m}
+\sum_{j< m\leq k,i_m=i} \frac{1}
{c_1^{a_{i_1,i}}\cd c_{m-1}^{a_{i_{m-1},i}}c_m}}
{\displaystyle\sum_{1\leq m<j,i_m=i} 
 \frac{c}
{c_1^{a_{i_1,i}}\cd c_{m-1}^{a_{i_{m-1},i}}c_m}+
\mathop\sum_{j\leq m\leq k,i_m=i}  \frac{1}
{c_1^{a_{i_1,i}}\cd c_{m-1}^{a_{i_{m-1},i}}c_m}},
\label{eici}\\
&& \vep_i(Y_{\bf i}(c_1,\cd,c_k))=
\sum_{1\leq m\leq k,i_m=i} \frac{1}
{c_1^{a_{i_1,i}}\cd c_{m-1}^{a_{i_{m-1},i}}c_m},
\label{vep-i}\\
&&\gamma_i(Y_{\bf i}(c_1,\cd,c_k))
=c_1^{a_{i_1,i}}\cd c_k^{a_{i_k,i}}.
\label{gamma-i}
\end{eqnarray}
{\sl Remark.}
As in \cite{N}, the above setting requires the condition
$I=\{i_1,\cd,i_k\}$.
Otherwise, set $J:=\{i_1,\cd,i_k\}\subsetneq I$ and let $\ge_J\subsetneq \ge$ 
be the corresponding subalgebra.
Then, by arguing similarly to \cite[4.3]{N}, we can define the $\ge_J$-geometric crystal 
structure on $B^-_{\bf i}$.
\subsection{Positive structure,\,\,
Ultra-discretizations \,\, and \,\,Tropicalizations}
\label{positive-str}

Let us recall the notions of 
positive structure, ultra-discretization and tropicalization.

The setting below is same as in \cite{KNO}.
Let $T=(\bbC^\times)^l$ be an algebraic torus over $\bbC$ and 
$X^*(T):={\rm Hom}(T,\bbC^\times)\cong \ZZ^l$ 
(resp. $X_*(T):={\rm Hom}(\bbC^\times,T)\cong \ZZ^l$) 
be the lattice of characters
(resp. co-characters)
of $T$. 
Set $R:=\bbC(c)$ and define
$$
\begin{array}{cccc}
v:&R\setminus\{0\}&\longrightarrow &\ZZ\\
&f(c)&\mapsto
&{\rm deg}(f(c)),
\end{array}
$$
where $\rm deg$ is the degree of poles at $c=\ify$. 
Here note that for $f_1,f_2\in R\setminus\{0\}$, we have
\begin{equation}
v(f_1 f_2)=v(f_1)+v(f_2),\q
v\left(\frac{f_1}{f_2}\right)=v(f_1)-v(f_2)
\label{ff=f+f}
\end{equation}
A non-zero rational function on
an algebraic torus $T$ is called {\em positive} if
it can be written as $g/h$ where
$g$ and $h$ are positive linear combinations of
characters of $T$.
\begin{df}
Let 
$f\cl T\rightarrow T'$ be 
a rational morphism between
two algebraic tori $T$ and 
$T'$.
We say that $f$ is {\em positive},
if $\eta\circ f$ is positive
for any character $\eta\cl T'\to \C$.
\end{df}
Denote by ${\rm Mor}^+(T,T')$ the set of 
positive rational morphisms from $T$ to $T'$.

\begin{lem}[\cite{BK}]
\label{TTT}
For any $f\in {\rm Mor}^+(T_1,T_2)$             
and $g\in {\rm Mor}^+(T_2,T_3)$, 
the composition $g\circ f$
is well-defined and belongs to ${\rm Mor}^+(T_1,T_3)$.
\end{lem}

By Lemma \ref{TTT}, we can define a category ${\mathcal T}_+$
whose objects are algebraic tori over $\bbC$ and arrows
are positive rational morphisms.

Let $f\cl T\rightarrow T'$ be a 
positive rational morphism
of algebraic tori $T$ and 
$T'$.
We define a map $\what f\cl X_*(T)\rightarrow X_*(T')$ by 
\[
\langle\eta,\what f(\xi)\rangle
=v(\eta\circ f\circ \xi),
\]
where $\eta\in X^*(T')$ and $\xi\in X_*(T)$.
\begin{lem}[\cite{BK}]
For any algebraic tori $T_1$, $T_2$, $T_3$, 
and positive rational morphisms 
$f\in {\rm Mor}^+(T_1,T_2)$, 
$g\in {\rm Mor}^+(T_2,T_3)$, we have
$\what{g\circ f}=\what g\circ\what f.$
\end{lem}
Let ${\hbox{\germ Set}}$ denote the category of sets with the morphisms being set maps.
By the above lemma, we obtain a functor: 
\[
\begin{array}{cccc}
{\mathcal{UD}}:&{\mathcal T}_+&\longrightarrow &{{\hbox{\germ Set}}}\\
&T&\mapsto& X_*(T)\\
&(f:T\rightarrow T')&\mapsto& 
(\what f:X_*(T)\rightarrow X_*(T')))
\end{array}
\]

\begin{df}[\cite{BK}]
Let $\chi=(X,\{e_i\}_{i\in I},\{{\rm wt}_i\}_{i\in I},
\{\vep_i\}_{i\in I})$ be a 
geometric crystal, $T'$ an algebraic torus
and $\theta:T'\rightarrow X$ 
a birational isomorphism.
The isomorphism $\theta$ is called 
{\it positive structure} on
$\chi$ if it satisfies
\begin{enumerate}
\item for any $i\in I$ the rational functions
$\gamma_i\circ \theta:T'\rightarrow \bbC$ and 
$\vep_i\circ \theta:T'\rightarrow \bbC$ 
are positive.
\item
For any $i\in I$, the rational morphism 
$e_{i,\theta}:\bbC^\tm \tm T'\rightarrow T'$ defined by
$e_{i,\theta}(c,t)
:=\theta^{-1}\circ e_i^c\circ \theta(t)$
is positive.
\end{enumerate}
\end{df}
Let $\theta:T\rightarrow X$ be a positive structure on 
a geometric crystal $\chi=(X,\{e_i\}_{i\in I},$
$\{{\rm wt}_i\}_{i\in I},
\{\vep_i\}_{i\in I})$.
Applying the functor ${\mathcal{UD}}$ 
to positive rational morphisms
$e_{i,\theta}:\bbC^\tm \tm T\rightarrow T$ and
$\gamma_i\circ \theta,\vep_i\circ\theta:T\ra \bbC$
(the notations are
as above), we obtain
\begin{eqnarray*}
\til e_i&:=&{\mathcal{UD}}(e_{i,\theta}):
\ZZ\tm X_*(T) \rightarrow X_*(T)\\
{\rm wt}_i&:=&{\mathcal{UD}}(\gamma_i\circ\theta):
X_*(T')\rightarrow \bbZ,\\
\vep_i&:=&{\mathcal{UD}}(\vep_i\circ\theta):
X_*(T')\rightarrow \bbZ.
\end{eqnarray*}
Now, for given positive structure $\theta:T'\rightarrow X$
on a geometric crystal 
$\chi=(X,\{e_i\}_{i\in I},$
$\{{\rm wt}_i\}_{i\in I},
\{\vep_i\}_{i\in I})$, we associate 
the quadruple $(X_*(T'),\{\til e_i\}_{i\in I},
\{{\rm wt}_i\}_{i\in I},\{\vep_i\}_{i\in I})$
with a free pre-crystal structure (see \cite[Sect.7]{BK}) 
and denote it by ${\mathcal{UD}}_{\theta,T'}(\chi)$.
We have the following theorem:

\begin{thm}[\cite{BK,N}]
For any geometric crystal 
$\chi=(X,\{e_i\}_{i\in I},\{\gamma_i\}_{i\in I},$
$\{\vep_i\}_{i\in I})$ and positive structure
$\theta:T'\rightarrow X$, the associated pre-crystal
${\mathcal{UD}}_{\theta,T'}(\chi)=$\\
$(X_*(T'),\{\til e_i\}_{i\in I},\{{\rm wt}_i\}_{i\in I},
\{\vep_i\}_{i\in I})$ 
is a crystal {\rm (see \cite[Sect.7]{BK})}
\end{thm}

Now, let ${\mathcal{GC}}^+$ be a category whose 
object is a triplet
$(\chi,T',\theta)$ where 
$\chi=(X,\{e_i\},\{\gamma_i\},\{\vep_i\})$ 
is a geometric crystal and $\theta:T'\rightarrow X$ 
is a positive structure on $\chi$, and morphism
$f:(\chi_1,T'_1,\theta_1)\longrightarrow 
(\chi_2,T'_2,\theta_2)$ is given by a rational map
$\vp:X_1\longrightarrow X_2$  
($\chi_i=(X_i,\cd)$) such that 
\begin{eqnarray*}
&&\vp\circ e^{X_1}_{i}=e^{X_2}_{i}\circ\vp,\q \gamma^{X_2}_{i}\circ\vp=\gamma^{X_1}_{i},\q
\vep^{X_2}_{i}\circ\vp=\vep^{X_1}_{i},\\
&&\text{ and }f:=\theta_2^{-1}\circ\vp\circ\theta_1:T'_1\longrightarrow T'_2,
\end{eqnarray*}
is a positive rational morphism. Let ${\mathcal{CR}}$
be the category of crystals. 
Then by the theorem above, we have
\begin{cor}
\label{cor-posi}
The map $ \mathcal{UD} = \mathcal{UD}_{\theta,T'}$ defined above is a functor
\begin{eqnarray*}
 {\mathcal{UD}}&:&{\mathcal{GC}}^+\longrightarrow {\mathcal{CR}},\\
&&(\chi,T',\theta)\mapsto X_*(T'),\\
&&(f:(\chi_1,T'_1,\theta_1)\rightarrow 
(\chi_2,T'_2,\theta_2))\mapsto
(\what f:X_*(T'_1)\rightarrow X_*(T'_2)).
\end{eqnarray*}

\end{cor}
We call the functor $\mathcal{UD}$
{\it ``ultra-discretization''} as in (\cite{N,N2})
instead of ``tropicalization'' as in \cite{BK}.
And 
for a crystal $B$, if there
exists a geometric crystal $\chi$ and a positive 
structure $\theta:T'\rightarrow X$ on $\chi$ such that 
${\mathcal{UD}}(\chi,T',\theta)\cong B$ as crystals, 
we call an object $(\chi,T',\theta)$ in ${\mathcal{GC}}^+$
a {\it tropicalization} of $B$, which is not standard but
we use such a terminology as before.

\section{Perfect Crystals of type $\TY(A,1,n)$}
\label{perf}

From now on we assume $\ge$ to be the affine Lie algebra $\TY(A,1,n), n \geq 2$. In this section, we recall the coherent family of  perfect crystals of type 
$\TY(A,1,n), n \geq 2$ and its limit given in (\cite{OSS}, \cite{KMN2}). For basic notions of crystals, coherent family of perfect crystals and its limit we refer the reader to \cite{KKM} 
(See also \cite{KMN1,KMN2}).

For the affine Lie algebra $A_n^{(1)}$, let 
$\{\alpha_0, \alpha_1, \cdots \alpha_n\}$, 
$\{\al^\vee_0, \al^\vee_1, \cdots  \al^\vee_n\}$ and 
$\{\Lm_0, \Lm_1,\\
 \cdots \Lm_n\}$ be the set of 
simple roots, simple coroots and fundamental weights, respectively.
The Cartan matrix $A=(a_{ij})_{i,j \in I}$ , $I = \{0, 1, \cdots , n\}$
is given by:
\begin{eqnarray*}
a_{ij}= \begin{cases}
2 \qquad &{\rm if} \quad i=j ,\\
-1 \qquad &{\rm if} \quad  i \equiv (j \pm 1) \ \ {\rm mod} (n+1),\\
0 \qquad &{\rm otherwise}
\end{cases}
\end{eqnarray*}
and its Dynkin diagram is as follows. 
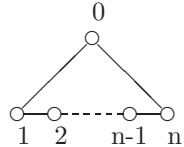
\begin{figure}[h!]
\begin{center}

\setlength{\unitlength}{.5cm}
\begin{picture}(5,2)
\linethickness{0.05mm}
\put(.175,0){\line(1,0){.65}}
\put(3.175,0){\line(1,0){.65}}
\put(.15,.15){\line(1,1){1.7}}
\put(2.15,1.85){\line(1,-1){1.7}}
\multiput(1.175,0)(.3,0){6}{\line(1,0){.15}}
\put(0,0){\circle{.35}}
\put(0,-.8){1}
\put(1,0){\circle{.35}}
\put(1,-.8){2}
\put(3,0){\circle{.35}}
\put(2.5,-.8){n-1}
\put(4,0){\circle{.35}}
\put(4,-.8){n}
\put(2,2){\circle{.35}}
\put(2,2.5){0}
\end{picture}

\end{center}
\label{Dynkin}
\end{figure}

The standard null root $\delta$ 
and the canonical central element $c$ are 
given by
\[
\delta=\alpha_0+\alpha_1+ \cdots +\alpha_n
\quad\text{and}\quad c=\al^\vee_0+\al^\vee_1+ \cdots +\al^\vee_n, 
\]
where 
$\al_0=2\Lm_0-\Lm_1-\Lm_n+\del,\q
\al_i=-\Lm_{i-1}+2\Lm_i-\Lm_{i+1}, 1 \le i \le n-1 \q
\al_n=-\Lm_0 - \Lm_{n-1}+2\Lm_n.$

For a positive integer $l$ we introduce 
$\TY(A,1,n)$-crystals $B^{2,l}$ and $B^{2,\ify}$ as 
\begin{eqnarray*}
&&B^{2,l}=\left\{
b=(b_{ji})_{1\le j \le 2 ,  j \le i \le j+n-1}
\left\vert
\begin{array}{c}
b_{ji} \in \ZZ_{\geq  0} , \sum_{i=j}^{j+n-1} b_{ji} = l , 1 \le j \le 2\\
\sum_{i=1}^t b_{1i} \geq \sum_{i=2}^{t+1} b_{2i} , 1 \le t \le n
\end{array}
\right.
\right\},\\
&&B^{2,\ify}=\left\{
b=(b_{ji})_{1\le j \le 2 ,  j \le i \le j+n-1}
\left\vert
b_{ji} \in \ZZ , \sum_{i=j}^{j+n-1} b_{ji} = 0 , 1 \le j \le 2
\right.
\right\}. 
\end{eqnarray*}
Now we describe the explicit crystal structures
of $B^{2,l}$ and $B^{2,\ify}$. 
Indeed, most of them coincide with 
each other except
for $\vep_0$ and $\vp_0$.
In the rest of this section, we use the following 
convention: 
$(x)_+=\max(x,0)$.
For $b=(b_{ji})$ we denote
\begin{equation} \label{z_i}
z_i= b_{1i} - b_{2,i+1}, \quad  2 \le i \le n-1.
\end{equation}

Now we define conditions ($E_m$) and ($F_m$) for $ 2 \le m \le n$  as follows.

\begin{equation} 
\label{(F)}
(F_m) : \q  \begin{cases} z_k + z_{k+1}+ \cdots + z_{m-1} \le 0 , &
2 \leq k \leq m-1\\
 z_m + z_{m+1} + \cdots + z_k > 0, &  m \le k \le n-1.
 \end{cases}
 \end{equation}
 \begin{equation}
\label{(E)} 
(E_m) : \q \begin{cases} z_k + z_{k+1}+ \cdots + z_{m-1} < 0 , &
2 \leq k \leq m-1\\
z_m + z_{m+1} + \cdots + z_k \geq 0, & m \le k \le n-1.
\end{cases}
\end{equation}
We also define 
\begin{equation} \label{Delta}
\Delta (m) = (b_{12}+b_{13}+ \cdots +b_{1,m-1})+(b_{2,m+1}+b_{2,m+2}+ \cdots +b_{2n}), \quad 2 \le m \le n.
\end{equation}

Let $\Delta = {\rm min} \{ \Delta(m) \mid 2 \le m \le n\}$. Note that for $2 \le m \le n$, $\Delta = \Delta (m)$ if the condition $(F_m)$ (or $(E_m)$) hold. Then for $b=(b_{ji}) \in B^{2,l}$ or $B^{2,\ify}$,
$\et{k}(b), \ft{k}(b), \vep_k(b), \vp_k(b), k= 0, 1, \cdots , n$ are given as follows. 

For $0 \le k \le n$, $\et{k}(b) = (b'_{ji})$, where

\begin{align*}
\begin{cases}
k= 0: &  b'_{11} = b_{11}-1,  b'_{1m} = b_{1m}+1, b'_{2m}= b_{2m}-1, b'_{2,n+1}=b_{2,n+1}+1\\ 
 & \text{ if } \ \ (E_m), 2 \le m \le n, \\
k= 1: & b'_{11} = b_{11}+1,  b'_{12} = b_{12}-1, \\ 
2 \le k \le n-1: &  \begin{cases}
b'_{1k} = b_{1k}+1, b'_{1,k+1} = b_{1,k+1}-1 \ \ \text{ if } \ \ b_{1k} \geq b_{2,k+1},\\
b'_{2k} = b_{2k}+1, b'_{2,k+1} = b_{2,k+1}-1 \ \ \text{ if } \ \  b_{1k} < b_{2,k+1},
\end{cases}\\
k= n: & b'_{2n}= b_{2n} + 1, \, \, b'_{2,n+1} = b_{2,n+1}-1
\end{cases}
\end{align*}

and $b'_{ji} = b_{ji}$ otherwise. 

For $0 \le k \le n$, $\ft{k}(b) = (b'_{ji})$, where

\begin{align*}
\begin{cases}
k= 0: &  b'_{11} = b_{11}+1,  b'_{1m} = b_{1m}-1, b'_{2m}= b_{2m}+1, b'_{2,n+1}=b_{2,n+1}-1\\ 
 & \text{ if } \ \ (F_m), 2 \le m \le n, \\
k= 1: & b'_{11} = b_{11}-1,  b'_{12} = b_{12}+1, \\ 
2 \le k \le n-1: &  \begin{cases}
b'_{1k} = b_{1k}-1, b'_{1,k+1} = b_{1,k+1}+1 \ \ \text{ if } \ \ b_{1k} > b_{2,k+1},\\
b'_{2k} = b_{2k}-1, b'_{2,k+1} = b_{2,k+1}+1 \ \ \text{ if } \ \  b_{1k} \le b_{2,k+1},
\end{cases}\\
k= n: & b'_{2n}= b_{2n} - 1, \, \, b'_{2,n+1} = b_{2,n+1} +1
\end{cases}
\end{align*}

and $b'_{ji} = b_{ji}$ otherwise. 
For $b\in B^{2,l}$ 
if $\et{k}b$ or $\ft{k}b$ does not belong to 
$B^{2,l}$ then
we understand it to be $0$.

\begin{align*}
\vep_1(b)=&b_{12},
\qq\vp_1(b)=
b_{11}-b_{22},
\\
\vep_k(b)=& b_{1,k+1}+(b_{2,k+1} - b_{1,k})_+  
\qq\vp_k(b)=b_{2k} + (b_{1k}-b_{2,k+1})_+,\\
&{\text {for}} \q 2 \le k \le n-1,\\
\vep_n(b) =& b_{2,n+1} - b_{1n}, \qq \vp_n(b) = b_{2n}\\
\vep_0(b)=&
\begin{cases}
l-b_{2,n+1} - \Delta, & b\in B^{2,l},\\
-b_{2,n+1} - \Delta, & b\in B^{2,\ify},
\end{cases}\\
\vp_0(b)=&
\begin{cases}
l-b_{11} - \Delta, & b\in B^{2,l},\\
-b_{11} - \Delta, & b\in B^{2,\ify}.
\end{cases}
\end{align*}

Hence the weights $wt_i(b) = \vp_i(b) - \vep_i(b), 0\le i \le n$ are:
\begin{align*}
\begin{cases}
wt_0(b) = b_{2, n+1} - b_{11},\\
wt_1(b) = b_{11}-b_{12}-b_{22},\\
wt_k(b) = (b_{1k} - b_{1,k+1})+(b_{2k} - b_{2,k+1})& (1< k<n),\\
wt_n(b) = b_{1n}+b_{2n}-b_{2,n+1}.
\end{cases}
\end{align*}
The following results have been proved in (\cite{KMN2}, \cite{OSS}):
\begin{thm}[\cite{KMN2,OSS}] 
\begin{enumerate}
\item The $\TY(A,1,n)$-crystal $B^{2,l}$ 
is a perfect crystal of level $l$. 
\item
The family of the perfect crystals 
$\{B^{2,l}\}_{\l\geq1}$ forms a 
coherent family and the crystal $B^{2,\ify}$ 
is its limit with the vector 
$b_\ify=(0)_{2\times n}$.
\end{enumerate}
\end{thm}

\renewcommand{\thesection}{\arabic{section}}
\section{Affine Geometric Crystal $\cV(\TY(A,1,n))$}
\setcounter{equation}{0}
\renewcommand{\theequation}{\thesection.\arabic{equation}}


Let $c=\sum_{i=0}^n \al^\vee_i$ be the canonical
central element in the affine Lie algebra $\ge = \TY(A,1,n)$ and
$\{\Lm_i|i\in I\}$ be the set of fundamental 
weights as in the previous section. Let $\sigma$ denote the Dynkin diagram 
automorphism. In particular, $\sigma (\Lm_i) = \Lm_{\overline{i+1}}$, where $\overline{i+1} = (i+1) \ {\rm mod} (n+1)$.
Consider the level $0$ fundamental weight $\varpi_2:=\Lm_2-\Lm_0$. Let $I_0 = I \setminus {0}, \ \ I_n = I \setminus {n}$, 
and $\ge _ i$ denote the subalgebra of $\ge$ associated with the index sets $I_i, i= 0, n$. Then $\ge_0$ as well as $\ge_n$ is 
isomorphic to $A_n$.

Let $W(\varpi_2)$ be the fundamental representation 
of $\uqp$ associated with $\varpi_2$ (\cite{K0}).
By \cite[Theorem 5.17]{K0}, $W(\varpi_2)$ is a
finite-dimensional irreducible integrable 
$\uqp$-module and has a global basis
with a simple crystal. Thus, we can consider 
the specialization $q=1$ and obtain the 
finite-dimensional $\TY(A,1,n)$-module $W(\varpi_2)$, 
which we call a fundamental representation
of $\TY(A,1,n)$ and use the same notation as above.
We shall present the explicit form of 
$W(\varpi_2)$ below.

\subsection{Fundamental representation 
$W(\varpi_2)$ for $\TY(A,1,n)$}
\label{fundamental}
The $\TY(A,1,n)$-module $W(\varpi_2)$ is an $\frac{1}{2}n(n+1)$-dimensional 
module with the basis,

\[
\{(i, j) \mid 1 \le i <j \le n+1\}, 
\]

where $(i, j)$ denotes the tableaux:
 
\begin{figure}[h!]
\begin{center}
\setlength{\unitlength}{.4cm}
\begin{picture}(1,2)
\put(0,1){\makebox(1,1)[c]{$i$}}
\put(0,0){\makebox(1,1)[c]{$j$}}
\linethickness{0.05mm}
\put(0,0){\line(1,0){1}}
\put(0,1){\line(1,0){1}}
\put(0,2){\line(1,0){1}}
\put(0,0){\line(0,1){2}}
\put(1,0){\line(0,1){2}}
\end{picture}
\end{center}
\label{ij boxes}
\end{figure}
The actions of $e_i$ and $f_i$ on these basis vectors
are given as follows.

For $1 \le k \le n$, we have
\begin{eqnarray*}
f_k (i,j) &=& \begin{cases}
(i+1, j), & i= k < j-1 \\
(i, j+1), & j= k \\
0, &\text{otherwise}.
\end{cases} \\
e_k (i,j) &=& \begin{cases}
(i-1, j), & i= k+1 \\
(i, j-1), & i < j-1= k \\
0, & \text{otherwise}.
\end{cases}
\end{eqnarray*}
\begin{eqnarray*}
f_0 (i,j) &=& \begin{cases}
(1, i), & i\not= 1,  j= n+1 \\
0, & \text{otherwise}.
\end{cases} \\
e_0 (1,j) &=&\begin{cases}
(j, n+1), & i\not= 1 \\
0, & \text{otherwise}.
\end{cases}
\end{eqnarray*}

Furthermore the weights of the basis vectors are given by:
\begin{eqnarray*}
wt(i, j) = (\Lm_i - \Lm_{i-1} + \Lm_j - \Lm_{j-1}) \qquad 1 \le i < j \le n+1,
\end{eqnarray*}
where we understand that $\Lm_{n+1} = \Lm_0$. Note that in $W(\varpi_2)$, we have
$(1 , 2)$ (resp. $(1 , n+1)$) is a $\ge_0$ (resp. $\ge_n$) highest weight vector
with weight $\varpi_2 = \Lm_2 - \Lm_0$ (resp. $\sigma^{-1}\varpi_2 = \Lm_1 - \Lm_n$).

\subsection{Affine Geometric Crystal $\cV(\TY(A,1,n))$ in $W(\varpi_2)$ }

Now we will construct the affine geometric crystal $\cV(\TY(A,1,n))$ in $W(\varpi_2)$ 
explicitly. For $\xi\in (\frt^*_{\rm cl})_0$, let $t(\xi)$ be the 
translation as in \cite[Sect 4]{K0} and $\wtil\varpi_i$ 
as in \cite{K1}. Indeed, 
$\wtil\varpi_i:=\max(1,\frac{2}{(\al_i,\al_i)})\varpi_i = \varpi_i$ in our case.
Then we have 
\begin{eqnarray*}
&& t(\wtil\varpi_2)=\sigma^2(s_{n-1}s_{n-2} \cdots s_1)(s_ns_{n-1} \cdots s_2)=:\sigma^2w_1,\\
&& t(\text{wt}(1 , n+1))=\sigma^2(s_{n-2}s_{n-3} \cdots s_0)(s_{n-1}s_{n-2} \cdots s_1)=:\sigma^2w_2,
\end{eqnarray*}
Associated with these Weyl group elements $w_1, w_2 \in W$,
we define algebraic varieties $\cV_1, \ \cV_2\subset W(\varpi_2)$ as follows.
\begin{eqnarray*}
&&\hspace{-30pt}\cV_1:=\{V_1(x)
:=Y_{n-1}(x_{2n-1}) \cdots Y_1(x_{n+1})Y_n(x_n) \cdots Y_2(x_2)
(1, 2)\,\,\vert\,\,x_i\in\bbC^\times \},\\
&&\hspace{-30pt}\cV_2:=\{V_2(y):=
Y_{n-2}(y_{2n-2}) \cdots Y_0(y_n)Y_{n-1}(y_{n-1}) \cdots Y_1(y_1)
(1, n+1)\,\,\vert\,\,y_i\in\bbC^\times \}.
\end{eqnarray*}
Using  the explicit actions of $f_i$'s on $W(\varpi_2)$
as above, we have $f_i^2=0$, for all $i \in I$. Therefore, 
we have 
\[
Y_i(c)=(1+\frac{f_i}{c})\al_i^\vee(c)
\,\,{\rm for \  all} \ \  i \in I.
\]
Thus we can get explicit forms of $V_1(x)\in\cV_1$ 
and $V_2(y)\in\cV_2$. Set
\begin{eqnarray*}
V_1(x) = &V_1(x_2, x_3, \cdots x_{2n-1}) = \sum_{1\leq i < j \leq n+1} X_{ij} (i, j), \\
V_2(y)=& V_2(y_1, y_2, \cdots y_{2n-2}) = \sum_{1\leq i < j \leq n+1} Y_{ij} (i, j).
\end{eqnarray*}
where the coefficients  $X_{ij}$'s and $Y_{ij}$'s can be computed explicitly.
These coefficients are positive rational functions in the variables $(x_2, \cd,x_{2n-1})$ and 
$(y_1, \cd,y_{2n-2})$ respectively and they are given as follows:
\begin{eqnarray*}
X_{ij} = \begin{cases}
x_{i+1} + \frac{x_{i+2}x_{n+i}}{x_{n+i+1}} +
	  \frac{x_{i+3}x_{n+i}}{x_{n+i+2}} + \cd +
	  \frac{x_{n}x_{n+i}}{x_{2n-1}}, & i\not=n, j=n \\ \displaystyle
x_{n+j}\left(x_{i+1} + \frac{x_{i+2}x_{n+i}}{x_{n+i+1}} + \frac{x_{i+3}x_{n+i}}{x_{n+i+2}} + \cd + \frac{x_{j}x_{n+i}}{x_{n+j-1}}\right), &
i\not=n, i+1 \leq j \leq n-1\\ \displaystyle
x_{n+i} , &  i \not= n, j= n+1 \\ \displaystyle
1 , & i=n, \, j= n+1.
\end{cases}
\end{eqnarray*}

\begin{eqnarray*}
Y_{ij} = \begin{cases}\displaystyle
y_{n+j}\left(y_{i+1} + \frac{y_{i+2}y_{n+i}}{y_{n+i+1}} + \frac{y_{i+3}y_{n+i}}{y_{n+i+2}} + \cd + \frac{y_jy_{n+i}}{y_{n+j-1}}\right), 
& 1 \leq i < j \leq n-2\\ \displaystyle
y_{i+1} + \frac{y_{i+2}y_{n+i}}{y_{n+i+1}} +
	  \frac{y_{i+3}y_{n+i}}{y_{n+i+2}} + \cd +
	  \frac{y_{n-1}y_{n+i}}{y_{2n-2}}, 
& 1 \leq i\leq n-2, j=n-1 \\ \displaystyle
y_{n+i} , & 1 \leq i \leq n-2 , j= n \\ \displaystyle
1 , &i=n-1, \, j= n \\
y_{n+i}\left(y_{1} + \frac{y_{2}y_{n}}{y_{n+1}} + \frac{y_{3}y_{n}}{y_{n+2}} + \cd + \frac{y_{i}y_{n}}{y_{n+i-1}}\right), 
& 1 \leq i  \leq n-2, j=n+1\\ \displaystyle
y_{1} + \frac{y_{2}y_{n}}{y_{n+1}} + \frac{y_{3}y_{n}}{y_{n+2}} + \cd + \frac{y_{n-1}y_{n}}{y_{2n-2}}, 
&i=  n-1, j=n+1\\ \displaystyle
y_n , & i=n, \, j= n+1.
\end{cases}
\end{eqnarray*}

Now for a given $x=(x_2, x_3, ,\cd, x_{2n-1})$ we solve the equation 
\begin{equation}
 V_2(y)=a(x)V_1(x),
\label{eq}
\end{equation}
where $a(x)$ is a rational function in $x=(x_2, x_3, ,\cd, x_{2n-1})$.
Though this equation is over-determined, 
it can be solved uniquely by direct calculation and the explicit form of 
solution is given below.

\begin{lem}
\label{x->y}
We have the rational function $a(x)$ and 
the unique solution of (\ref{eq}):
\begin{eqnarray*}\displaystyle
&& \displaystyle a(x) = \frac{1}{x_n},\, \,
y_1= \left(\frac{x_{2}}{x_{n+1}} + \frac{x_3}{x_{n+2}} + \cd + \frac{x_n}{x_{2n-1}}\right)^{-1}, \\
&& \displaystyle y_k= x_k\left(\frac{x_{k+1}}{x_{n+k}} + \frac{x_{k+2}}{x_{n+k+1}} + \cd + \frac{x_n}{x_{2n-1}}\right)^{-1}, \, 2 \leq k \leq n-1,\\
&& \displaystyle 
y_n= \frac{1}{x_n}, \, \, y_{n+l}
= \frac{x_{n+l}}{x_n}\left(\frac{x_{l+1}}{x_{n+l}} 
+ \frac{x_{l+2}}{x_{n+l+1}} + \cd + 
\frac{x_n}{x_{2n-1}}\right), \, 1 \leq l \leq n-2.
\end{eqnarray*}
\end{lem}

Now using Lemma \ref{x->y} we define the map 
\begin{eqnarray*}
\ovl\sigma:& \cV_1 \qq\qq \longrightarrow &\cV_2,\\
&V_1(x_2,\cd,x_{2n-1})\mapsto &V_2(y_1,\cd,y_{2n-2}).
\end{eqnarray*}
Then we have the following result.
\begin{pro}
The map $\ovl\sigma:\cV_1\longrightarrow \cV_2$ is a bi-positive birational isomorphism with the inverse
positive rational map
\begin{eqnarray*}
\ovl\sigma^{-1}:&\cV_2\qq\qq \longrightarrow &\cV_1,\\
&V_2(y_1,\cd,y_{2n-2})\mapsto &V_1(x_2,\cd,x_{2n-1}).
\end{eqnarray*} 
given by:
\begin{eqnarray*}
&&\displaystyle x_k= \frac{y_k}{y_n}\left(\frac{y_{1}}{y_{n}} + \frac{y_2}{y_{n+1}} + \cd + \frac{y_k}{y_{n+k-1}}\right)^{-1}, \, 2 \leq k \leq n-1,\\
&&\displaystyle x_{n+l}= y_{n+l}\left(\frac{y_{1}}{y_{n}} + \frac{y_{2}}{y_{n+1}} + \cd + \frac{y_l}{y_{n+l-1}}\right), \, 1 \leq k \leq n-2,\\
&&\displaystyle x_n = \frac{1}{y_n}, \, \, x_{2n-1}= \left(\frac{y_{1}}{y_{n}} + \frac{y_{2}}{x_{n+1}} + \cd + 
\frac{y_{n-1}}{y_{2n-2}}\right).
\end{eqnarray*}
\end{pro}
\begin{proof}
The fact that $\ovl\sigma$ is a bi-positive birational map follows from the explicit formulas. The rest 
follows by direct calculation.
\end{proof}

It is known (see \cite{KNO} and 2.3) that $\cV_1$ (resp. $\cV_2$) is a geometric crystal
for $\ge_0$ (resp. $\ge_n$). Indeed,  we have the $\ge_0$-geometric
crystal structure on $\cV_1$ by setting
$Y(x)=Y(x_{2n-1},\cd,x_2):=Y_{n-1}(x_{2n-1})\cd Y_{2}(x_2)$, 
$V_1(x)=V_1(x_{2n-1},\cd,x_2):=Y(x)(1,2)$ and 
\[
e_i^c(V_1(x)):=e_i^c(Y(x))(1,2),\q
\gamma_i(V_1(x))=\gamma_i(Y(x)),\q
\vep_i(V_1(x)):=\vep_i(Y(x)),
\]
since the vector $(1,2)$ is the highest weight vector with respect to $\ge_0$.
Similarly, we obtain the $\ge_n$-geometric crystal structure on $\cV_2$.
Hence the actions of $e_i^c , \gamma_i,  \vep_i$ (resp. 
$\ovl{e}_i^c , \ovl{\gamma}_i,  \ovl{\vep}_i)$
on $V_1(x)$ (resp. $V_2(y)$) are described explicitly for $i \in I_0$
(resp. $i \in I_n$) by the formula in 2.3. In particular, 
the actions of $\ovl{e}_0^c , \ovl{\gamma}_0$ and $\ovl{ \vep}_0$ on $V_2(y)$ are given by:
\begin{eqnarray*}
&&\displaystyle \ovl{e}_0^c(V_2(y)) = V_2(y_1, \cd , cy_n, \cd , y_{2n-2}) , \\
&&\displaystyle \ovl{\gamma}_0(V_2(y)) = \frac{y_n^2}{y_1y_{n+1}} , \qquad \ovl{\vep}_0(V_2(y)) = \frac{y_{n+1}}{y_n}.
\end{eqnarray*}
In order to make 
$\cV_1$ a $A_n^{(1)}$- geometric crystal we need to define the actions of $e_0^c , \gamma_0$ 
and $ \vep_0$ on $V_1(x)$. 
We define the action of $e_0^c$ on $V_1(x)$ by
\begin{equation}
e_0^cV_1(x)=\ovl\sigma^{-1}\circ \ovl{e}_0^c\circ
\ovl\sigma(V_1(x))).
\label{e0}
\end{equation}
and the actions of  $\gamma_0$ 
and $\vep_0$  on $V_1(x)$ by 
\begin{equation}
\gamma_0(V_1(x))=\ovl{ \gamma}_0( \ovl\sigma(V_1(x))),\qq
\vep_0(V_1(x)):=\ovl{\vep}_0 (\ovl\sigma(V_1(x))).
\label{wt0}
\end{equation}
\begin{thm}
Together with the actions of $ e_0^c, \gamma_0$ and $\vep_0$ on $V_1(x)$ given in (\ref{e0}), 
(\ref{wt0}), we obtain a
positive affine geometric crystal $\cV(\TY(A,1,n)):=
(\cV_1,\{e_i\}_{i\in I},
\{\gamma_i\}_{i\in I},\\
\{\vep_i\}_{i\in I})$
$(I=\{0,1, \cd ,n\})$, whose explicit form is as follows:
first we have $e_i^c(V_1(x))$, $\gamma_i(V_1(x))$ and $\vep_i(V_1(x))$
for $i=1,2, \cd , n$ from the formula (\ref{eici}), 
(\ref{vep-i})
and (\ref{gamma-i}).
\begin{eqnarray*}
e_i^c(V_1(x)) = \begin{cases}
V_1(x_2, \cd , cx_{n+1}, \cd , x_{2n-1}) , \qq i = 1,\\
V_1(x_2, \cd , c_ix_i, \cd , \frac{c}{c_i}x_{n+i}, \cd , 
x_{2n-1}), \, \, 2 \leq i \leq n-1,\\
V_1(x_2, \cd , cx_n, \cd , x_{2n-1}), \qq i=n
\end{cases}
\end{eqnarray*}
where 
\[
c_i = \frac{c(x_ix_{n+i}+x_{i+1}x_{n+i-1})}{cx_ix_{n+i}+x_{i+1}x_{n+i-1}}.
\]

\begin{eqnarray*}
\gamma_i(V_1(x)) = \begin{cases}
\displaystyle \frac{x_{n+1}^2}{x_2x_{n+2}} , & i = 1,\\
\displaystyle \frac{x_i^2x_{n+i}^2}{x_{i-1}x_{i+1}x_{n+i-1}x_{n+i+1}} , &
 2 \leq i \leq n-1,\\
\displaystyle \frac{x_n^2}{x_{n-1}x_{2n-1}} , & i=n.
\end{cases}
\end{eqnarray*}

\begin{eqnarray*}
\vep_i(V_1(x)) = \begin{cases}
\displaystyle \frac{x_{n+2}}{x_{n+1}} , & i = 1,\\
\displaystyle \frac{x_{n+i+1}}{x_{n+i}} +
 \frac{x_{i+1}x_{n+i-1}x_{n+i+1}}{x_{i}x_{n+i}^2} , 
& 2 \leq i \leq n-2,\\
\displaystyle \frac{1}{x_{2n-1}} + 
\frac{x_{n}x_{2n-2}}{x_{n-1}x_{2n-1}^2} , & i = n-1,\\
\displaystyle \frac{x_{2n-1}}{x_{n}} , & i=n.
\end{cases}
\end{eqnarray*}

Using (\ref{e0}) and (\ref{wt0}),  
the explicit actions of  $e_0^c$, $\vep_0$ and $\gamma_0$ on $V_1(x)$
are given by:
\begin{eqnarray*}
&&\gamma_0(V_1(x))= \frac{1}{x_nx_{n+1}}, \, \, \, 
\vep_0(V_1(x))=x_{n+1} \left(\frac{x_2}{x_{n+1}} + \frac{x_3}{x_{n+2}} +
\cd +  \frac{x_n}{x_{2n-1}}\right),\\
&&e_0^c(V_1(x))= V_1(x') = V_1(x_2', x_3', \cd , x_{2n-1}'),
\end{eqnarray*}
\text{where}\\
\begin{eqnarray*}
\begin{cases}
\displaystyle x_k' = x_k\cdot\frac{\frac{x_2}{x_{n+1}}
+ \frac{x_3}{x_{n+2}} + \cd 
+ \frac{x_n}{x_{2n-1}}}{c\left(\frac{x_2}{x_{n+1}}
+\frac{x_3}{x_{n+2}} + \cd + \frac{x_k}{x_{n+k-1}}\right) 
+ \left(\frac{x_{k+1}}{x_{n+k}} + \cd +
 \frac{x_n}{x_{2n-1}}\right)}\, , 
& 2\leq k <n,
\\
\displaystyle x_n' = \frac{x_n}{c}, \qq x_{n+1}' = \frac{x_{n+1}}{c},\\
\displaystyle x_{n+l}' = x_{n+l} \cdot
\frac{c\left(\frac{x_2}{x_{n+1}} + \frac{x_3}{x_{n+2}} 
+ \cd + \frac{x_l}{x_{n+l-1}}\right) 
+ \left(\frac{x_{l+1}}{x_{n+l}} 
+ \cd + \frac{x_n}{x_{2n-1}}\right)}{c\left(\frac{x_2}{x_{n+1}} 
+ \frac{x_3}{x_{n+2}} + \cd + 
\frac{x_n}{x_{2n-1}}\right)}\,, & 2 \leq l <n.
\end{cases}
\end{eqnarray*}
\end{thm}
\begin{proof}
Since the positivity is clear from the explicit formulas, 
it suffices to show that $\cV(\TY(A,1,n)):=
(V_1(x),\{e^c_i\}_{i\in I}, \{\gamma_i\}_{i\in I},\{\vep_i\}_{i\in I})$ 
satisfies the relations in Definition (\ref{def-gc}). 
Indeed, since $\cV_1$ is a $\ge_0$ geometric crystal 
we need to check the relations involving the $0$-index:
\begin{enumerate}
\item[(1)] $\gamma_0(e_i^c(V_1(x))) = c^{a_{i0}} \gamma_0(V_1(x)), 1 \leq i \leq n, $
\item[(2)] 
$\gamma_i(e_0^c(V_1(x))) = c^{a_{0i}} \gamma_i(V_1(x)), \, \,1\leq i  \leq  n, $
\item[(3)] $\vep_0(e_0^c(V_1(x))) = c^{-1}\vep_0(V_1(x)), $
\item[(4)] $e_0^ce_1^{cd}e_0^d = e_1^de_0^{cd}e_1^c, $
\item[(5)] $e_0^ce_n^{cd}e_0^d = e_n^de_0^{cd}e_n^c , $
\item[(6)]$e_0^ce_i^d = e_i^de_0^c, \, \, 2 \leq i \leq n-1.$
\end{enumerate}
Since 
\begin{eqnarray*}
\gamma_0(e_i^c(V_1(x))) = \begin{cases}
\displaystyle \frac{c^2}{x_nx_{n+1}}, & i = 0, \\
\displaystyle \frac{1}{cx_nx_{n+1}}, & i = 1, n , \\
\displaystyle \frac{1}{x_nx_{n+1}}, & 2 \leq i \leq n-1,
\end{cases}
\end{eqnarray*}
and
\begin{eqnarray*}
\gamma_i(e_0^c(V_1(x))) = \begin{cases}
\displaystyle \frac{x_{n+1}^2}{cx_nx_{n+2}}, & i = 1, \\
\displaystyle \frac{x_n^2}{cx_{n-1}x_{2n-1}}, & i = n , \\
\displaystyle 
\frac{x_i^2x_{n+i}^2}{x_{i-1}x{i+1}x_{n+i-1}x_{n+i+1}}, &2 \leq i \leq n-1,
\end{cases}
\end{eqnarray*} 
we have (1) and (2) hold. We also have (3) hold since 
$\cV_2$ is a $\ge_n$-geometric crystal and hence
\begin{eqnarray*}
&&\vep_0(e_0^c(V_1(x))) = \ovl{\vep}_0\ovl\sigma\ovl\sigma^{-1}\ovl{e}_0^c\ovl\sigma(V_1(x)) 
= \ovl{\vep}_0\ovl{e}_0^c(V_2(y))\\
&& = \ovl{\vep}_0(V_2(y')) = \frac{y_{n+1}'}{y_n'} = \frac{y_{n+1}}{cy_n}
=c^{-1}\vep_0(V_1(x)).
\end{eqnarray*} 
By direct calculations we see that on $V_1(x)$ we have
\begin{eqnarray*}
\ovl{\sigma} \circ e_i^c = \ovl{e}_i^c \circ \ovl{\sigma}, \, \, \text{for} \qq 1 \leq i \leq n-1.
\end{eqnarray*}
Hence for $2 \leq i \leq n-1$, we have
\begin{eqnarray*}
&&e_0^ce_i^d = (\ovl{\sigma}^{-1}\ovl{e}_0^c\ovl{\sigma})(\ovl{\sigma}^{-1}\ovl{e}_i^d\ovl{\sigma})
= \ovl{\sigma}^{-1}\ovl{e}_0^c\ovl{e}_i^d\ovl{\sigma} \\ 
&&= \ovl{\sigma}^{-1}\ovl{e}_i^d\ovl{e}_0^c\ovl{\sigma} = e_i^de_0^c ,
\end{eqnarray*}
and 
\begin{eqnarray*}
&&e_0^ce_1^{cd}e_0^d = (\ovl{\sigma}^{-1}\ovl{e}_0^c\ovl{\sigma})(\ovl{\sigma}^{-1}\ovl{e}_1^{cd}\ovl{\sigma})
(\ovl{\sigma}^{-1}\ovl{e}_0^d\ovl{\sigma}) \\
&&= \ovl{\sigma}^{-1}\ovl{e}_0^c\ovl{e}_i^{cd}\ovl{e}_0^d\ovl{\sigma} = \ovl{\sigma}^{-1}\ovl{e}_1^d\ovl{e}_0^{cd}\ovl{e}_1^c\ovl{\sigma} = e_1^de_0^{cd}e_1^c ,
\end{eqnarray*}
since $\cV_2$ is a $\ge_n$-geometric crystal. Therefore, (4) and (6)
 hold. 

Now for $k=2,\cd,n-1$ we set $X=X_k+\wtil X_k$ where 
\[
 X_k=\frac{x_2}{x_{n+1}}
+ \frac{x_3}{x_{n+2}} + \cd 
+ \frac{x_k}{x_{k+n-1}},\qq
\wtil X_k=\frac{x_{k+1}}{x_{k+n}}
+ \frac{x_{k+2}}{x_{k+n+1}} + \cd 
+ \frac{x_n}{x_{2n-1}}.
\]
Observe that for any $k,l=2,\cd,n-1$ we have $X=X_k+\wtil X_k=X_l+\wtil X_l$.
Recall that $e_0^c(V_1(x))=V_1(x')= V_1(x'_2,\cd,x'_{2n-1})$. Now we have 
\begin{equation}
\frac{x'_k}{x'_{k+n-1}}=
\frac{cX^2}{c-1}\left(\frac{1}{cX_{k-1}+\wtil X_{k-1}}-
\frac{1}{cX_{k}+\wtil X_{k}}\right)\q
(3\leq k\leq n-1,\,\,c\ne1).
\label{k-formula}
\end{equation}
Using Equation(\ref{k-formula}) we can easily see that (5) holds which
completes the proof.
\end{proof}
\renewcommand{\thesection}{\arabic{section}}
\section{Ultra-discretization of $\cV(\TY(A,1,n))$}
\setcounter{equation}{0}
\renewcommand{\theequation}{\thesection.\arabic{equation}}

We denote the positive structure on $\cV = \cV(\TY(A,1,n))$ as in the 
previous section by 
$\theta:T'\seteq(\bbC^\times)^{2n-2} \longrightarrow \cV$
($x\mapsto V_1(x)$).
Then by Corollary \ref{cor-posi}
we obtain the ultra-discretization 
$\cX = {\mathcal{UD}}(\cV,T',\theta)$ 
which is a Kashiwara's crystal. 
Now we show that the conjecture in \cite{KNO} 
holds for $\ge = \TY(A,1,n), \, i = 2$ by giving an explicit 
isomorphism of crystals between $\cX$ and $B^{2, \infty}$.
In order to show this isomorphism, we need
the explicit crystal structure on 
$\cX:={\mathcal{UD}}(\chi,T',\theta)$.
Note that $\cX = \ZZ^{2n-2}$ as a set .
In $\cX$, we use the same notations $c,x_0,x_2,\cd,x_{2n-1}$
for variables as in  $\cV$.

For $x=(x_2,x_1,\cd,x_{2n-1})\in\cX$, by applying the ultra-discretization functor
$\mathcal{UD}$ it follows from the 
results in the previous section that the functions
$\wt_i= \mathcal{UD}(\gamma_i), \, \vep_i = \mathcal{UD}(\vep_i)$ and $\mathcal{UD}(e_i^c)$ for
 $i=0,1, \cd , n$ are given by:
\begin{eqnarray*}
\wt_i(x) = \begin{cases}
-x_n-x_{n+1}, & i=0,\\
-x_2+2x_{n+1}-x_{n+2}, & i=1,\\
2x_2-x_3-x_{n+1}+2x_{n+2}-x_{n+3}, & i= 2, \\
-x_{i-1}+2x_i-x_{i+1}-x_{n+i-1}+2x_{n+i}-x_{n+i+1}, & 3\leq i <n,\\
-x_{n-1}+2x_n-x_{2n-1}, & i=n.
\end{cases}
\end{eqnarray*}
\begin{eqnarray*}
\vep_i(x) = \begin{cases}
x_{n+1}+\text{max}_{2\leq k \leq n}(\beta_k), & i=0,\\
-x_{n+1}+x_{n+2}, & i=1,\\
\text{max} (x_{n+i+1}-x_{n+i}, -x_i+x_{i+1}+x_{n+i-1}-2x_{n+i}+x_{n+i+1}), &
 2\leq i \leq n-2,\\
\text{max}(-x_{2n-1}, -x_{n-1}+x_n+x_{2n-2}-2x_{2n-1}), & i=n-1,\\
-x_n+x_{2n-1}, & i=n,
\end{cases}
\end{eqnarray*}
where $\beta_k:= x_k-x_{n+k-1}$ for $2 \leq k \leq n$.
\begin{eqnarray*}
\mathcal{UD}(e_i^c)(x)= \begin{cases}
(x_2+C_2, \cd , x_{n-1}+C_{n-1},x_n-c,x_{n+1}-c,\\
x_{n+2}-c-C_2, \cd , x_{2n-1}-c-C_{n-1}),& i=0, \\
(x_2, \cd , x_n,x_{n+1}+c,x_{n+2}, \cd , x_{2n-1}), & i=1,\\
(x_2, \cd , x_i+\ovl{c}_i, \cd , x_{n+i}+c-\ovl{c}_i, \cd , x_{2n-1}), 
& 2\leq i <n,\\
(x_2, \cd , x_{n-1},x_n+c,x_{n+1}, \cd , x_{2n-1}), & i= n,
\end{cases}
\end{eqnarray*}
where
\begin{eqnarray*}
&&C_k= \text{max}_{2\leq j\leq n}(\beta_j)
-\text{max}(\text{max}_{2\leq j\leq k}(c+\beta_j) , 
 \text{max}_{k < j\leq n}(\beta_j)), \, 2 \leq k <n,\\
&& \ovl{c}_i=c + \text{max}(x_i+x_{n+i}, x_{i+1}+x_{n+i-1}) 
- \text{max}(c+x_i+x_{n+i}, x_{i+1}+x_{n+i-1}),\,\, 2\leq i <n.
\end{eqnarray*}
Note that the Kashiwara operators are $\et{i}(x) = {\mathcal{UD}}e_i^c(x)\mid_{c=1}$ and 
$\ft{i}(x) =$ \\ $ {\mathcal{UD}}e_i^c(x)\mid_{c=-1}$ on $\cX$. In particular, for $x \in \cX$, we have
\begin{eqnarray}
\label{f1n}
 \begin{cases}
\ft{1}(x) = (x_2, \cd , x_{n+1}-1, \cd , x_{2n-1}),\\
\ft{n}(x) = (x_2, \cd , x_n-1, \cd , x_{2n-1}),
\end{cases} 
\end{eqnarray}
and for $2 \leq i \leq n-1$,
\begin{eqnarray}
\label{fi}
\ft{i}(x) = \begin{cases}
(x_2, \cd ,  x_{n+i}-1, \cd , x_{2n-1}), \qq \text{if} \, \, \, \beta_i > \beta_{i+1},\\
(x_2, \cd , x_i-1, \cd , x_{2n-1}), \qq \text{if} \, \, \, \beta_i \leq \beta_{i+1}.
\end{cases} 
\end{eqnarray}
To determine the explicit action of $\ft{0}$ we define conditions:
\begin{eqnarray}
\label{phi}
(\phi_j) : \qq \beta_2, \cd , \beta_{j-1} \leq \beta_j > \beta_{j+1}, \cd , \beta_n
\end{eqnarray}
for each $2 \leq j \leq n$ where we assume $\beta_1 = 0 = \beta_{n+1}$.  Note that
under condition $(\phi_j)$ we have:
\begin{eqnarray*}
C_2 =  \cd = C_{j-1} = 0, \, \, \, \text{and}  \, \, \, C_j = \cd = C_{n-1} = 1.
\end{eqnarray*}
Hence for $x \in \cX$ and $2 \leq j \leq n$ we have 
\begin{eqnarray*}
&&\ft{0}(x) = (x_2, \cd , x_{j-1}, x_j+1, x_{j+1}+1, \cd , x_{n+j-1}+1, x_{n+j}, \cd , x_{2n-1}), \\
&&\text{if} \, \,  \text{condition} \, \,  (\phi_j) \, \,  \text{hold}.
\end{eqnarray*}
\begin{thm}
\label{ultra-d}
The map 
\[
\begin{array}{cccc}
\Omega\cl&\cX&\longrightarrow& B^{2,\infty},\\
&(x_2,\cd, x_{2n-1})&\mapsto&  b=(b_{ji})_{1\leq j\leq 2, j\leq i \leq j+n-1},
\end{array}
\]
defined by
\begin{eqnarray*}
&& b_{11}= x_{n+1}, \, \, \, b_{1i}= x_{n+i}-x_{n+i-1}, \, \, \, 2\leq i \leq n-1,\, \, \,
b_{1n}= -x_{2n-1}, \\
&& b_{22}= x_2, \, \, \,  b_{2i}=x_i-x_{i-1}, \, \, \, 3\leq i \leq n, \, \, \, 
b_{2,n+1}= -x_n,
\end{eqnarray*}
is an isomorphism of crystals.
\end{thm} 

\begin{proof} First we observe that the map $\Omega^{-1}\cl B^{2,\infty} \longrightarrow \cX$ is given by
$\Omega^{-1}(b) = x = (x_2, \cd , x_{2n-1})$ where
\begin{eqnarray*}
x_i = \sum_{k=2}^i b_{2k}, \qq 2\leq i \leq n, \\
x_{n+i} = \sum_{k=1}^ib_{1k} , \qq 1 \leq i \leq n-1.
\end{eqnarray*}
Hence the map $\Omega$ is bijective. To prove that $\Omega$ is an isomorphism of crystals we need to 
show that it commutes with the actions of $\ft{i}$ and preserves the actions of the functions $\wt_i$ and $\vep_i$. In particular we 
need to show that for $x \in \cX$ and $0 \leq i \leq n$ we have:
\begin{eqnarray*}
&&\Omega (\ft{i}(x)) = \ft{i}(\Omega(x)), \\
&&\wt_i(\Omega(x))=\wt_i(x), \\
&&\vep_i(\Omega(x))=\vep_i(x).
\end{eqnarray*} 
Indeed commutativity of $\Omega$ and $\et{i}$ follows similarly. 
For $x \in \cX$, set $\Omega(x) = b = (b_{ji}) \in B^{2, \infty}$.
First let us check $\wt_i$.
\begin{eqnarray*}
&&\hspace{-10pt} \wt_0(\Omega(x)) = \wt_0(b) = b_{2,n+1}-b_{11} = -x_n -x_{n+1} = \wt_0(x).\\
&&\hspace{-10pt} \wt_1(\Omega(x)) = \wt_1(b) = b_{11}-b_{12} - b_{22} = x_{n+1} -(x_{n+2}- x_{n+1})-x_2\\
&& =-x_2 +2x_{n+1} -x_{n+2} = \wt_1(x).\\
&&\hspace{-10pt} \wt_2(\Omega(x)) = \wt_2(b) = (b_{12}-b_{13}) - (b_{22} - b_{23}) \\
&&= x_{n+2} -x_{n+1}- x_{n+3}+ x_{n+2} +x_2-x_3+x_2\\
&&= 2x_2 - x_3 -x_{n+1} +2x_{n+2} -x_{n+3} = \wt_2(x).\\
&&\hspace{-10pt} \wt_i(\Omega(x)) = \wt_i(b) = (b_{1i}-b_{1,i+1}) +(b_{2i} - b_{2,i+1}) \\
&&= x_{n+i}-x_{n+i-1} -x_{n+i+1}+ x_{n+i}+x_i-x_{i-1}-x_{i+1}+x_i \\
&&=-x_{i-1}+2x_i-x_{i+1}-x_{n+i-1}+2x_{n+i}-x_{n+i+1} = \wt_i(x), \, \, \, \, 3 \leq i \leq n-1.\\
&&\hspace{-10pt} \wt_n(\Omega(x)) = \wt_n(b) = b_{1n} + (b_{2n} - b_{2,n+1}) \\
&&= -x_{2n-1} +x_n- x_{n-1}+ x_n =-x_{n_1}+2x_n-x_{2n-1}= \wt_n(x).\\
\end{eqnarray*}
Next, we shall check $\vep_i$:
\begin{eqnarray*}
&&\hspace{-10pt} \vep_0(\Omega(x)) = \vep_0(b) = -b_{2,n+1}-\Delta \\
&&=-b_{2,n+1}-\text{min}_{2\leq k \leq n}(b_{12}+ \cd +b_{1,k-1}+b_{2,k+1}+ \cd + b_{2n}) \\
&&=x_n-\text{min}_{2\leq k \leq n}(x_{n+k-1}-x_{n+1}+x_n-x_k) \\
&&=x_n+\text{max}_{2\leq k \leq n}(-x_{n+k-1}+x_{n+1}-x_n+x_k) \\
&&= x_{n+1}+ \text{max}(x_k - x_{n+k-1})= \vep_0(x).\\
&&\hspace{-10pt} \vep_1(\Omega(x)) = \wt_1(b) = b_{12} = x_{n+2} - x_{n+1} = \vep_1(x).
\end{eqnarray*}
\begin{eqnarray*}
&&\hspace{-10pt} \vep_i(\Omega(x)) = \vep_i(b) = b_{1, i+1} +(b_{2,i+1} - b_{1i})_+ \\
&&= \text{max} (b_{1,i+1}, b_{1,i+1}+b_{2,i+1}-b_{1i}) \\
&&=-\text{max}( x_{n+i+1}-x_{n+i} , -x_i+x_{i+1}+x_{n+i-1}-2x_{n+i}+x_{n+i+1})= \vep_i(x), \\
&& \text{for} \qq 2 \leq i \leq n-2.\\
&&\hspace{-10pt} \vep_{n-1}(\Omega(x)) = \vep_{n-1}(b) = \text{max}(b_{1n} , b_{1n}+b_{2n}-b_{1,n-1}) \\
&&= \text{max}(-x_{2n-1} ,  - x_{n-1}+ x_n+x{2n-2}-2x_{2n-1}) = \vep_{n-1}(x).\\
&&\hspace{-10pt} \vep_{n}(\Omega(x)) = \vep_{n}(b) = b_{2,n+1} - b_{1n} = -x_n+x_{2n-1} = \vep_n(x). 
\end{eqnarray*}
Now we shall check that $\Omega(\ft{i}(x)) = \ft{i}(\Omega(x))$ for $i= 0, 1, \cd , n$.
$$
\ft{1}(\Omega(x)) = \ft{1}(b) = b' = (b'_{ji}),
$$
where
\begin{eqnarray*}
b'_{11}= b_{11}-1 = x_{n+1} - 1, \, \, b'_{12}= b_{12}+1 = x_{n+2} - x_{n+1} + 1, \, \, b'_{ji} = b_{ji}, \text{otherwise}.
\end{eqnarray*}
Hence $\Omega(\ft{1}(x)) = \Omega(x_2, \cd , x_{n+1} - 1, \cd , x_{2n-1}) = \ft{1}(\Omega(x))$.

$$
\ft{n}(\Omega(x)) = \ft{n}(b) = b' = (b'_{ji}),
$$
where
\begin{eqnarray*}
b'_{2n}= b_{2n}-1 = x_n-x_{n-1} - 1, \, \, b'_{2,n+1}= b_{2,n+1}+1 = - x_n + 1, \, \, b'_{ji} = b_{ji}, \text{otherwise}.
\end{eqnarray*}
Hence $\Omega(\ft{n}(x)) = \Omega(x_2, \cd , x_n - 1, \cd , x_{2n-1}) = \ft{n}(\Omega(x))$.
Now we check that $\Omega(\ft{i}(x)) = \ft{i}(\Omega(x))$ for $2 \leq i \leq n-1$.
Let $\ft{i}(\Omega(x)) = \ft{i}(b) = b' = (b'_{ji})$. Note that  $b_{1i} = x_{n+i} - x_{n+i-1}$ and
$b_{1,i+1} = x_{i+1} - x_{i}$. Hence $b_{1i} > b_{2,i+1}$ (resp. $b_{1i} \leq b_{2,i+1}$) if and only if
$\beta_i > \beta_{i+1}$ (resp.  $\beta_i \leq \beta_{i+1}$).

If $x_{n+i} - x_{n+i-1} > x_{i+1} - x_{i}$, then $\ft{i}(\Omega(x)) = \ft{i}(b) = b' = (b'_{ji})$, where
\begin{eqnarray*}
&&b'_{1i}= b_{1i}-1 = x_{n+i}-x_{n+i-1} - 1, \, \, b'_{1,i+1}= b_{1,i+1}+1 = x_{n+i+1}-x_{n+i} + 1, \\
&&b'_{ji} = b_{ji}, \text{otherwise}.
\end{eqnarray*}
Hence $\Omega(\ft{i}(x)) = \Omega(x_2, \cd , x_{n+i}-1, \cd ,x_{2n-1}) = \ft{i}(\Omega(x))$ in this case.

If $x_{n+i} - x_{n+i-1} \leq x_{i+1} - x_{i}$, then $\ft{i}(\Omega(x)) = \ft{i}(b) = b' = (b'_{ji})$, where
\begin{eqnarray*}
&&b'_{2i}= b_{2i}-1 = x_i-x_{i-1} - 1, \, \, b'_{2,i+1}= b_{2,i+1}+1 = x_{i+1}-x_i + 1, \\
&&b'_{ji} = b_{ji}, \text{otherwise}.
\end{eqnarray*}
Hence $\Omega(\ft{i}(x)) = \Omega(x_2, \cd , x_i-1, \cd ,x_{2n-1}) = \ft{i}(\Omega(x))$ in this case.

Finally we want to verify that $\Omega(\ft{0}(x)) =  \ft{0}(\Omega(x))$. For $2 \leq m \leq n$, we have
$\ft{0}(\Omega(x)) = \ft{0}(b) = b' = (b'_{ji})$ where 
\begin{eqnarray*}
&&b'_{11} = b_{11}+1= x_{n+1}+1, \\
&& b'_{1m} = b_{1m}-1= \begin{cases} x_{n+m}-x_{n+m-1}-1, \, \, \text{if} \, \, \, m\not= n\\
-x_{2n-1}-1, \, \, \text{if} \, \, \, m = n \end{cases},\\
&&b'_{2m}= b_{2m}+1=\begin{cases} x_2 +1, \, \, \text{if} \, \, \, m=2\\
x_m - x_{m-1} +1 \, \, \text{if} \, \, \, m \not= 2 \end{cases}, \\
&& b'_{2,n+1}=b_{2,n+1}-1= -x_n-1, \qq b'_{ji} = b_{ji}, \, \, \text{otherwise},
\end{eqnarray*}
if the condition ($F_m$) in (\ref{(F)}) holds. 
Since $z_i = b_{1i} - b_{2, i+1} = (x_{n+i} - x_{n+i-1}) - (x_{i+1} -
 x_i) = \beta_i - \beta_{i+1}$
for $2 \leq i \leq n-1$, we observe that for 
$2 \leq m \leq n$, the condition ($F_m$) in 
(\ref{(F)}) holds if and only if the condition ($\phi_m$) in (\ref{phi}) holds.
Therefore, for $2 \leq m \leq n$, we have 
\begin{eqnarray*}
\Omega(\ft{0}(x))& =& \Omega(x_2, \cd , x_{m-1}, x_m+1, \cd , x_{n+m-1}+1, x_{n+m}, \cd , x_{2n-1}) \\
&=& \ft{0}(\Omega(x)),
\end{eqnarray*}
which completes the proof.

\end{proof}

\bibliographystyle{amsalpha}

\end{document}